\documentclass[12pt,fleqn]{amsart}
\usepackage{amsfonts}
\usepackage{latexsym,amssymb,amsmath,amsthm,amscd,graphicx}
\usepackage{color}
\usepackage{geometry}
\usepackage{accents}

\numberwithin{equation}{section}
\newtheorem{theorem}{Theorem}[section]
\newtheorem{lemma}[theorem]{Lemma}
\newtheorem{corollary}[theorem]{Corollary}

\theoremstyle{definition}

\newtheorem{definition}[theorem]{Definition}

\theoremstyle{remark}

\definecolor{blue}{rgb}{0,0,0.45}
\definecolor{red}{rgb}{0.7,0,0}
\input{tcilatex}

\begin{document}
\title[exponential approximation on the real line]{exponential approximation
of functions in Lebesgue spaces with Muckenhoupt weight}
\author{Ramazan Akg\"{u}n}
\maketitle

\begin{quotation}
\textbf{Abstract} Using a transference result, several inequalities of
approximation by entire functions of exponential type in $\mathcal{C}(%
\mathbf{R})$, the class of bounded uniformly continuous functions defined on 
$\mathbf{R}:=\left( -\infty ,+\infty \right) $, are extended to the Lebesgue
spaces $L^{p}\left( \mathbf{\varrho }dx\right) $ $1\leq p<\infty $ with
Muckenhoupt weight $\mathbf{\varrho }$ ($1\leq p<\infty $). This gives us a
different proof of Jackson type direct theorems and Bernstein-Timan type
inverse estimates in $L^{p}\left( \mathbf{\varrho }dx\right) $. Results also
cover the case $p=1$.

\textit{K}-functional.

\end{quotation}

\section{Introduction}

An entire function $f(z)$ is called of exponential type $\sigma \in \lbrack
0,\infty )$ (briefly e.f.e.t$\leq \sigma $) if%
\begin{equation*}
\limsup\nolimits_{|z|=r\rightarrow \infty }(r^{-1}\ln (\left(
\max\nolimits_{|z|=r}|f(z)|\right) )\leq \sigma \text{.}
\end{equation*}%
Sometimes an e.f.e.t$\leq \sigma $ is also recalled as band-limited
functions. See e.g. papers \cite{arsc,hjssic}.

Studies on e.f.e.t$\leq \sigma $ has intensified with problems related to
approximation of non-periodic continuous functions (see Bernstein's paper 
\cite{B12} of the year 1912) defined on the real axis $\mathbf{R}:=\left(
-\infty ,+\infty \right) $.

It is well known that trigonometric polynomials are not suitable apparatus
of approximation for non-periodic functions defined on $\mathbf{R}$ but the
class of e.f.e.t$\leq \sigma $ serve as a correct class for non-periodic
functions on $\mathbf{R}$.

Consider the class $A_{\mathbf{p}}$ of Muckenhoupt's weights \cite{Bm72}.
Main aim of this paper is to obtain central inequalities of approximation by
e.f.e.t$\leq \sigma $ for functions given in the Lebesgue spaces $%
L^{p}\left( \mathbf{\varrho }dx\right) $ on $\mathbf{R}$ with Muckenhoupt's
weights $\mathbf{\varrho \in }A_{\mathbf{p}}$, $1\leq p<\infty $.

Before statements of main results we can give some historical remarks and
achievements in the particular case of non-weighted classical Lebesgue
spaces $L^{p}(\mathbf{R})\mathbf{:=}L^{p}\left( \mathbf{\varrho }dx\right) $
with $\mathbf{\varrho \equiv }1$ and $1\leq p\leq \infty $.

After the results of S. N. Bernstein in \cite{B12}, some systematic studies
on e.f.e.t$\leq \sigma $ continued, chronologically, by R. P. Boas Jr. \cite%
{RPB}, A. F. Timan \cite{AFT}, N. I. Akhieser \cite{Ack}, S. M. Nikolski 
\cite{SMNbook}, I. I. Ibragimov \cite{II3}, H. Triebel \cite{HT83}, P. L.
Butzer, W. Splettst\"{o}$\beta $er and R. L. Stens \cite{BSS88}, H. N.
Mhaskar \cite{maskh}, H. J. Schmeisser and W. Sickel \cite{hjssic}, R. M.
Trigub and E. S. Belinsky \cite{TB04}. All these reference books contain
several inequalities of e.f.e.t$\leq \sigma $ in spaces $L^{p}(\mathbf{R})$
with $1\leq p\leq \infty $.

On the other hand, many other works also include results of approximation by
e.f.e.t$\leq \sigma .$ See, for example, papers M. F. Timan \cite{MFT61}-%
\cite{MFT}; S. M. Nikolski \cite{SMN51}; R. Taberski \cite{T81}, \cite{T86};
A. A. Ligun and V. D. Doronin \cite{LD}; V. Yu. Popov \cite{Po}; F. G.
Nasibov \cite{FGN}; A. I. Stepanets \cite{Ste1,Ste2}; S. B. Vakarchuk \cite%
{vak1},\cite{vak2},\cite{vak3}; V. G. Ponomarenko \cite{ponom}; G.
Gaimnazarov \cite{gaim1}; V. V. Arestov \cite{vva}, A. G. Babenko \cite{babe}%
; C. N. Vasil'ev \cite{CNV}; A. Guven and V. Kokilashvili \cite{GK}; Akgun
and Ghorbanalizadeh \cite{AG}; S. Artamonov, K. Runovski, H. J. Schmeisser 
\cite{arsc,ar1}; S. Artamonov \cite{art}; Z. Ditzian and A. Prymak \cite{DP}%
; P.-Pych Taberska and R. Taberski \cite{TT}; D. P. Dryanov, M. A. Qazi, and
Q. I. Rahman \cite{DQR03}; Z. Ditzian and K. V. Runovski \cite{DiRu}; Z.
Ditzian K. G. Ivanov \cite{Dit}; Y. Kolomoitsev and S. Yu. Tikhonov \cite%
{KolTik}; F. Dai, Z. Ditzian and S. Yu. Tikhonov \cite{DDTi}; Z. Burinska,
K. Runovski and H. J. Schmeisser \cite{BuRuSc06},\cite{BuRuSc}.

For the weighted Lebesgue spaces $L^{p}\left( \mathbf{\varrho }dx\right) ,$ $%
1\leq p\leq \infty ,$ with Freud type weights $\mathbf{\varrho }$ on $%
\mathbf{R}$ see book of H. Maskhar \cite{maskh} and paper of Z. Ditzian and
D. S. Lubinsky \cite{DL}. Recently, in several papers \cite{GIv}, \cite%
{GIvTi}, \cite{GK}, \cite{GorIva}, \cite{GorIvTikhon} of D.V. Gorbachev,
V.I. Ivanov, and S.Yu. Tikhonov it is studied approximation by spherical
e.f.e.t$\leq \sigma $ for functions given in $L^{p}\left( \mathbf{\varrho }%
dx\right) ,$ $1\leq p<\infty ,$ with Dunkl weights $\mathbf{\varrho }$ on $%
\mathbf{R}$.

After these historical remarks we can return the case of $L^{p}\left( 
\mathbf{\varrho }dx\right) $ $1\leq p<\infty $ with Muckenhoupt weight $%
\mathbf{\varrho }$ on $\mathbf{R}$.

For periodic $\mathbf{\varrho \in }A_{\mathbf{p}}$, $1<p<\infty $ and
periodic $f\in L^{p}\left( \mathbf{\varrho }dx\right) $, some results on
trigonometric approximation are known. See e.g. PhD thesis \cite{Gad} of 
\'{E}. A. Gadjieva; papers of N. X. Ky \cite{Ky1},\cite{Ky2}; S. Z. Jafarov 
\cite{Ja1}, \cite{ja2}; Author \cite{AK1}, \cite{eja}, \cite{ra11u}, \cite%
{spbu}, \cite{raTjm}; A. Guven and V. Kokilashvili \cite{GKstev}; Y. E.
Yildirir and D. M. Israfilov \cite{yeydmi11}; F. Abdullaev, A. Shidlich and
S. Chaichenko \cite{ascs}, and A.H. Av\c{s}ar and H. Ko\c{c} \cite{ahak}.

Recently author proved in \cite{Ak17} a transference result to obtain norm
inequalities for periodic functions in $L^{p}\left( \mathbf{\varrho }%
dx\right) ,$ $1\leq p<\infty ,$ $\mathbf{\varrho \in }A_{\mathbf{p}}.$

In the present work we will deal with non-periodic case in $L^{p}\left( 
\mathbf{\varrho }dx\right) ,$ $1\leq p<\infty ,$ $\mathbf{\varrho \in }A_{%
\mathbf{p}}.$

Let $\mathbb{N}$:$\mathbb{=}\left\{ 1,2,3,\cdots \right\} $ be the natural
numbers and $\mathbb{N}_{0}$:=$\mathbb{N\cup }\left\{ 0\right\} $.

For $j\in \mathbb{N}$, all constants $\mathbb{C}_{j}:=\mathbb{C}_{j}\left(
a,b,\cdots \right) $ will be some positive number such that they depend on
the parameters $a,b,\cdots $ and change only when parameters $a,b,\cdots $
change. Absolute constants will be denoted by $c_{i}>0$ ($i\in \mathbb{N}$)
and they will not change in each occurrences.

\section{Preliminary Notations and Transference Result}

A function $\mathbf{\varrho }:\mathbf{R}\mathbb{\rightarrow }\left[ 0,\infty %
\right] $ will be called weight if $\mathbf{\varrho }$ is measurable and
positive a.e. on $\mathbf{R}.$ Define $\left\langle \mathbf{\varrho }%
\right\rangle _{\mathbf{A}}:=\int_{\mathbf{A}}\mathbf{\varrho }(t)dt$ for $%
\mathbf{A}\subset \mathbf{R}$. A weight $\mathbf{\varrho }$ belongs to the
Muckenhoupt class $A_{\mathbf{p}}$, $1\leq p<\infty $, if%
\begin{equation}
\left\vert J\right\vert ^{-1}\left\langle \mathbf{\varrho }\right\rangle
_{J}\leq \left[ \mathbf{\varrho }\right] _{1}(\limfunc{essinf}%
\nolimits_{x\in J\text{ }}\mathbf{\varrho }\left( x\right) )\text{, a.e. on }%
\mathbf{R}\text{,\quad }\left( p=1\right) \text{,}  \label{a1123}
\end{equation}%
\begin{equation}
\left[ \mathbf{\varrho }\right] _{p}:=\sup\limits_{J\subset \mathbf{R}%
}\left\vert J\right\vert ^{-p}\left\langle \mathbf{\varrho }\right\rangle
_{J}\left\langle \mathbf{\varrho }^{1/(1-p)}\right\rangle _{J}^{p-1}<\infty 
\text{,\quad }\left( 1<p<\infty \right)   \label{ap}
\end{equation}%
with some finite constants independent of $J$.

For a weight $\mathbf{\varrho }$ on $\mathbf{R}$, we denote by $L^{p}\left( 
\mathbf{\varrho }dx\right) $, $1\leq p\leq \infty $ the class of real valued
measurable functions, defined on $\mathbf{R},$ such that%
\begin{equation*}
\left\Vert f\right\Vert _{p,\mathbf{\varrho }}\text{:=}\left( \int\nolimits_{%
\mathbf{R}}\left\vert f\left( x\right) \right\vert ^{p}\mathbf{\varrho }%
\left( x\right) dx\right) ^{1/p}\infty \text{,\quad (}1<p<\infty \text{) and}
\end{equation*}%
\begin{equation*}
\left\Vert f\right\Vert _{\infty ,\mathbf{\varrho }}\text{:=}%
ess.\sup\nolimits_{x\in \mathbf{R}}\left\vert f\left( x\right) \right\vert 
\text{,\quad (}p=\infty \text{).}
\end{equation*}

Let $C(\mathbf{R})$ (respectively $\mathcal{C}(\mathbf{R})$ ) be the class
of continuous (bounded uniformly continuous ) functions defined on $\mathbf{R%
}$. We denote by $C_{c}$ (respectively $S_{c}$) the collection of real
valued continuous (respectively simple) functions $f$ on $\mathbf{R}$ such
that support $sptf$ of $f$ is a compact set in $\mathbf{R}.$

For $1<p<\infty $, we\textbf{\ }set $(1/p)+(1/p^{\prime })=1$.

Proof of Lemma \ref{onL1}, given below, follows from \cite[(2.7) p.933 and
(2.10) p.934]{bg03}.

\begin{lemma}
\label{onL1}(\cite{bg03})If $p\in \lbrack 1,\infty )$, $\mathbf{\varrho }\in
A_{\mathbf{p}}$, and $f\chi _{A}\in L^{p}\left( \mathbf{\varrho }dx\right) $
then%
\begin{equation*}
\left\Vert f\chi _{A}\right\Vert _{1}\leq \left[ \mathbf{\varrho }\right]
_{p}^{1/p}\left\langle \mathbf{\varrho }\right\rangle _{A}^{-1/p}\left\Vert
f\chi _{A}\right\Vert _{p,\mathbf{\varrho }}
\end{equation*}%
holds for any compact subset $A$ of $\mathbf{R}$.
\end{lemma}

\begin{lemma}
If $1\leq p<\infty $, $\mathbf{\varrho \in }A_{\mathbf{p}}$, $f\in
L^{p}\left( \mathbf{\varrho }dx\right) $ and $g\in L^{p^{\prime }}(\mathbf{%
\varrho }dx)$, then, H\"{o}lder's inequality%
\begin{equation*}
\int\nolimits_{\mathbf{R}}\left\vert f(x)g(x)\right\vert \mathbf{\varrho }%
\left( x\right) dx\leq \left\Vert f\right\Vert _{p,\mathbf{\varrho }%
}\left\Vert g\right\Vert _{p^{\prime },\mathbf{\varrho }}
\end{equation*}%
holds.
\end{lemma}

\begin{definition}
Suppose that $0<\lambda <\infty $ and $\tau \in \mathbf{R}$. We define
family of translated Steklov operators $\{\mathbf{S}_{\lambda ,\tau }f\}$, by%
\begin{equation}
\mathbf{S}_{\lambda ,\tau }f(x):=\lambda \int\nolimits_{x+\tau -1/(2\lambda
)}^{x+\tau +1/(2\lambda )}f\left( t\right) dt,\quad x\in \mathbf{R}
\label{steklR}
\end{equation}%
for$\mathcal{\ }$locally integrable function $f$ defined on $\mathbf{R}$.
\end{definition}

\begin{definition}
Let $1\leq p<\infty $, $\mathbf{\varrho }\in A_{\mathbf{p}}$, $f\in
L^{p}\left( \mathbf{\varrho }dx\right) $. For $u\in \mathbf{R}$ we define%
\begin{equation}
F_{f}\left( u\right) \text{:=}\int\nolimits_{\mathbf{R}}\mathbf{S}%
_{1,u}f(x)\left\vert G(x)\right\vert \mathbf{\varrho }\left( x\right)
dx\quad \text{(}1\leq p<\infty \text{)}  \label{efef}
\end{equation}%
with $G\in L^{p^{\prime }}\left( \mathbf{\varrho }dx\right) $ satisfying $%
\left\Vert G\right\Vert _{p^{\prime },\varrho }\leq 1\mathbf{.}$
\end{definition}

\begin{definition}
(a) A family $Q$ of measurable sets $E\subset \mathbb{R}$ is called locally $%
N$-finite ($N\in \mathbb{N}$) if 
\begin{equation*}
\sum_{E\in Q}\chi _{E}\left( x\right) \leq N
\end{equation*}%
almost everywhere in $\mathbb{R}$ where $\chi _{U}$ is the characteristic
function of the set $U$.

(b) A family $Q$ of open bounded sets $U\subset \mathbb{R}$ is locally $1$%
-finite if and only if the sets $U\in Q$ are pairwise disjoint.
\end{definition}

\begin{theorem}
\label{stek}We suppose that $1\leq p<\infty $ and $\mathbf{\varrho }\in
A_{p} $. Then, family of Steklov Mean Operators $\{\mathbf{S}_{1,\tau
}\}_{\tau \in \mathbf{R}}$ \ is uniformly bounded (in $\tau $) in $%
L^{p}\left( \mathbf{\varrho }dx\right) $, namely,%
\begin{equation*}
\left\Vert \mathbf{S}_{1,\tau }f\right\Vert _{p,\mathbf{\varrho }}\leq
3\cdot 3^{\frac{2}{p}}\left[ \mathbf{\varrho }\right] _{p}^{1/p}\left\Vert
f\right\Vert _{p,\mathbf{\varrho }}\text{ \ \ for }\tau \in \mathbf{R}\text{.%
}
\end{equation*}
\end{theorem}

\begin{proof}[\textbf{Proof of Theorem \protect\ref{stek}}]
Let $Q$ be $1$-finite family of open bounded subsets $P_{i}$ of $\mathbb{R}$
having Lebesgue measure $1$ such that $\left( \cup _{i}P_{i}\right) \cup A=%
\mathbf{R}$ for some null-set $A$. Since $\tau \in \mathbf{R}$ there exists $%
m\in \mathbb{Z}$ such that $m\leq \tau <(m+2)$. Let $P+m$ be translation of
the set $P$ by $m$. We set $\left( P_{i}+m\right) ^{\pm }:=\left(
P_{i-1}\cup P_{i}\cup P_{i+1}\right) +m$. Then%
\begin{equation*}
\left\Vert \mathbf{S}_{1,\tau }f\right\Vert _{p,\mathbf{\varrho }%
}^{p}=\sum\limits_{P_{i}\in Q}\int\limits_{P_{i}}\left\vert \lambda
\int\limits_{x+\tau -1/2}^{x+\tau +1/2}f(t)dt\right\vert ^{p}\mathbf{\varrho 
}(x)dx
\end{equation*}%
\begin{equation*}
\leq \sum\limits_{P_{i}\in Q}\int\limits_{P_{i}}\left[ \int\limits_{x+\tau
-1/\lambda }^{x+\tau +1/2}\mathbf{\varrho }^{\frac{1}{p}}(t)\left\vert
f(t)\right\vert \frac{1}{\mathbf{\varrho }^{\frac{1}{p}}(t)}dt\right] ^{p}%
\mathbf{\varrho }(x)dx
\end{equation*}%
\begin{equation*}
\leq \sum\limits_{P_{i}\in Q}\int\limits_{P_{i}}\left( \left(
\int\limits_{x+\tau -1/2}^{x+\tau +1/2}\mathbf{\varrho }(t)\left\vert
f(t)\right\vert ^{p}dt\right) ^{\frac{1}{p}}\left( \int\limits_{x+\tau
-1/2}^{x+\tau +1/2}\mathbf{\varrho }^{\frac{-p^{\prime }}{p}}(t)dt\right) ^{%
\frac{1}{p^{\prime }}}\right) ^{p}\mathbf{\varrho }(x)dx
\end{equation*}%
\begin{equation*}
=\sum\limits_{P_{i}\in Q}\int\limits_{P_{i}}\int\limits_{x+\tau
-1/2}^{x+\tau +1/2}\mathbf{\varrho }(t)\left\vert f(t)\right\vert
^{p}dt\left( \int\limits_{x+\tau -1/2}^{x+\tau +1/2}\mathbf{\varrho }^{\frac{%
-p^{\prime }}{p}}(t)dt\right) ^{\frac{p}{p^{\prime }}}\mathbf{\varrho }(x)dx
\end{equation*}%
\begin{equation*}
=\sum\limits_{P_{i}\in Q}\int\limits_{P_{i}}\int\limits_{x+\tau
-1/2}^{x+\tau +1/2}\mathbf{\varrho }(t)\left\vert f(t)\right\vert
^{p}dt\left( \int\limits_{x+\tau -1/2}^{x+\tau +1/2}\mathbf{\varrho }^{-%
\frac{1}{p-1}}(t)dt\right) ^{p-1}\mathbf{\varrho }(x)dx
\end{equation*}%
\begin{equation*}
=\sum\limits_{P_{i}\in Q}\int\limits_{P_{i}}\left( \int\limits_{x+\tau
-1/2}^{x+\tau +1/2}\mathbf{\varrho }^{-\frac{1}{p-1}}(t)dt\right)
^{p-1}\left( \int\limits_{x+\tau -1/2}^{x+\tau +1/2}\mathbf{\varrho }%
(t)\left\vert f(t)\right\vert ^{p}dt\right) \mathbf{\varrho }(x)dx
\end{equation*}%
\begin{equation*}
\leq 3^{p}\sum\limits_{P_{i}\in Q}\left( \frac{1}{\left\vert P_{i}^{\pm
}\right\vert }\int\limits_{P_{i}^{\pm }}\mathbf{\varrho }(x)dx\right) \left( 
\frac{1}{\left\vert \left( P_{i}+m\right) ^{\pm }\right\vert }%
\int\limits_{\left( P_{i}+m\right) ^{\pm }}\mathbf{\varrho }^{-\frac{1}{p-1}%
}(t)dt\right) ^{p-1}\times
\end{equation*}%
\begin{equation*}
\times \int\limits_{\left( P_{i}+m\right) ^{\pm }}\mathbf{\varrho }%
(t)\left\vert f(t)\right\vert ^{p}dt
\end{equation*}%
\begin{equation*}
\leq 3^{p}\left[ \mathbf{\varrho }\right] _{p}\sum\limits_{P_{i}\in
Q}\left\{
\int\limits_{P_{i-1}+m}+\int\limits_{P_{i}+m}+\int\limits_{P_{i+1}+m}\right%
\} \mathbf{\varrho }(t)\left\vert f(t)\right\vert ^{p}dt
\end{equation*}%
\begin{equation*}
\leq 3^{p+1}\left[ \mathbf{\varrho }\right] _{p}\int\limits_{\mathbf{R}}%
\mathbf{\varrho }(t)\left\vert f(t)\right\vert ^{p}\left\{
\sum\limits_{P_{i}\in Q}\chi _{P_{i-1}+m}\left( t\right)
+\sum\limits_{P_{i}\in Q}\chi _{P_{i}+m}\left( t\right)
+\sum\limits_{P_{i}\in Q}\chi _{P_{i+1}+m}\left( t\right) \right\} dt
\end{equation*}%
\begin{equation*}
=3^{p+2}\left[ \mathbf{\varrho }\right] _{p}\int\limits_{\mathbf{R}}\mathbf{%
\varrho }(t)\left\vert f(t)\right\vert ^{p}dt=3^{p+2}\left[ \mathbf{\varrho }%
\right] _{p}\left\Vert f\right\Vert _{p,\mathbf{\varrho }}^{p}\text{ and}
\end{equation*}%
\begin{equation*}
\left\Vert \mathbf{S}_{1,\tau }f\right\Vert _{p,\mathbf{\varrho }}\leq
3\cdot 3^{\frac{2}{p}}\left[ \mathbf{\varrho }\right] _{p}^{1/p}\left\Vert
f\right\Vert _{p,\mathbf{\varrho }}.
\end{equation*}

For $p=1$, we find%
\begin{equation*}
\left\Vert \mathbf{S}_{1,\tau }f\right\Vert _{1,\mathbf{\varrho }%
}=\sum\limits_{P_{i}\in Q}\int\limits_{P_{i}}\left\vert \int\limits_{x+\tau
-1/2}^{x+\tau +1/2}f(t)dt\right\vert \mathbf{\varrho }(x)dx
\end{equation*}%
\begin{equation*}
\leq \sum\limits_{P_{i}\in Q}\int\limits_{P_{i}}\int\limits_{x+\tau
-1/2}^{x+\tau +1/2}\mathbf{\varrho }(t)\left\vert f(t)\right\vert \frac{1}{%
\mathbf{\varrho }(t)}dt\mathbf{\varrho }(x)dx
\end{equation*}%
\begin{eqnarray*}
&\leq &3\sum\limits_{P_{i}\in Q}\frac{1}{\left\vert \left( P_{i}+m\right)
^{\pm }\right\vert }\int\limits_{\left( P_{i}+m\right) ^{\pm }}\mathbf{%
\varrho }(x)dx\left( \underset{t\in \left( P_{i}+m\right) ^{\pm }}{\limfunc{%
esssup}}\frac{1}{\mathbf{\varrho }(t)}\right) \int\limits_{\left(
P_{i}+m\right) ^{\pm }}\mathbf{\varrho }(t)\left\vert f(t)\right\vert dt \\
&\leq &3\left[ \gamma \right] _{1}\sum\limits_{P_{i}\in Q}\left\{
\int\limits_{P_{i-1}+m}+\int\limits_{P_{i}+m}+\int\limits_{P_{i+1}+m}\right%
\} \mathbf{\varrho }(t)\left\vert f(t)\right\vert dt
\end{eqnarray*}%
\begin{equation*}
\leq 3\left[ \gamma \right] _{1}\int\limits_{\mathbf{R}}\mathbf{\varrho }%
(t)\left\vert f(t)\right\vert \left\{ \sum\limits_{P_{i}\in Q}\chi
_{P_{i-1}+m}\left( t\right) +\sum\limits_{P_{i}\in Q}\chi _{P_{i}+m}\left(
t\right) +\sum\limits_{P_{i}\in Q}\chi _{P_{i+1}+m}\left( t\right) \right\}
dt
\end{equation*}%
\begin{equation*}
\leq 9\left[ \gamma \right] _{1}\left\Vert f\right\Vert _{1,\mathbf{\varrho }%
}.
\end{equation*}%
Hence, for any $1\leq p<\infty $,%
\begin{equation*}
\left\Vert \mathbf{S}_{1,\tau }f\right\Vert _{p,\mathbf{\varrho }}\leq
3\cdot 3^{\frac{2}{p}}\left[ \mathbf{\varrho }\right] _{p}^{1/p}\left\Vert
f\right\Vert _{p,\mathbf{\varrho }}.
\end{equation*}
\end{proof}

\begin{theorem}
\label{du} Let $1\leq p<\infty $ and $\mathbf{\varrho }$ be a weight on $%
\mathbf{R}$. Then, for $f\in L^{p}\left( \mathbf{\varrho }dx\right) $ we have%
\begin{equation}
\underset{G\in L^{p^{\prime }}\left( \mathbf{\varrho }dx\right) ,\left\Vert
G\right\Vert _{p^{\prime },\mathbf{\varrho }}\leq 1}{\sup }\int\nolimits_{%
\mathbf{R}}\left\vert f\left( x\right) G\left( x\right) \right\vert \mathbf{%
\varrho }(x)dx\text{=}\left\Vert f\right\Vert _{p,\mathbf{\varrho }}.
\label{Cdual}
\end{equation}%
In addition, condition "$G\in L^{p^{\prime }}\left( \mathbf{\varrho }%
dx\right) $" in supremum can be replaced by condition "$G\in L^{p^{\prime
}}\left( \mathbf{\varrho }dx\right) \cap S_{c}$"
\end{theorem}

\begin{proof}[\textbf{Proof of Theorem \protect\ref{du}}]
(\ref{Cdual}) is a consequence of Theorem 18.4 of \cite{yeh}. On the other
hand, methods given in Lemma 2.7.2 and Lemma 3.2.14 of \cite{dhhr11} imply
that condition "$G\in L^{p^{\prime }}\left( \mathbf{\varrho }dx\right) $" in
supremum can be replaced by condition "$G\in L^{p^{\prime }}\left( \mathbf{%
\varrho }dx\right) \cap S_{c}$"
\end{proof}

\begin{theorem}
\label{tra} Let $1\leq p<\infty $, $\mathbf{\varrho }\in A_{\mathbf{p}}$,
and $f,g\in L^{p}\left( \mathbf{\varrho }dx\right) $. In this case,

(a) The function $F_{f,G}\left( \cdot \right) $ defined in (\ref{efef}) is
bounded, uniformly continuous on $\mathbf{R}$.

(b) If $\ \left\Vert F_{f}\right\Vert _{\mathcal{C}\left( \mathbf{R}\right)
}\leq c_{1}\left\Vert F_{g}\right\Vert _{\mathcal{C}\left( \mathbf{R}\right)
}$ \ holds with an absolute constant $c_{1}>0$, then, we have weighted norm
inequalities%
\begin{equation}
\left\Vert f\right\Vert _{p,\mathbf{\varrho }}\leq c_{1}\mathbb{C}%
_{1}\left\Vert g\right\Vert _{p,\mathbf{\varrho }}  \label{zz1}
\end{equation}%
with $\mathbb{C}_{1}:=\mathbb{C}_{1}\left( p,\mathbf{\varrho }\right)
:=6\cdot 3^{\frac{2}{p}}\left[ \mathbf{\varrho }\right] _{p}^{1/p}.$
\end{theorem}

\begin{proof}[\textbf{Proof of Theorem \protect\ref{tra}}]
(a) Since $C_{c}$ is a dense subset (\cite[proof of Theorem 4.1, item I]%
{koklSamko}) of $L^{p}\left( \mathbf{\varrho }dx\right) $, we consider
functions $H\in C_{c}$ and prove that \newline$F_{H}\left( u\right) $=$%
\int\nolimits_{\boldsymbol{R}}(\mathsf{S}_{1,u}H)\left( x\right) \left\vert
G\left( x\right) \right\vert \mathbf{\varrho }\left( x\right) dx$ is bounded
and uniformly continuous on $\boldsymbol{R}$, where $G\in L^{p^{\prime
}}\left( \mathbf{\varrho }dx\right) \cap S_{c}$ and $\left\Vert G\right\Vert
_{p^{\prime },\mathbf{\varrho }}\leq 1$. Boundedness of $F_{H}\left( \cdot
\right) $ is easy consequence of the H\"{o}lder's inequality and Theorem \ref%
{stek}. On the other hand, note that $H$ is uniformly continuous on $%
\boldsymbol{R}$, see e.g. Lemma 23.42 of \cite[pp.557-558]{yeh}. Take $%
\varepsilon >0$ and $u_{1},u_{2},x\in \boldsymbol{R}$. Then, there exists a $%
\delta :=\delta \left( \varepsilon \right) >0$ such that%
\begin{equation*}
\left\vert H\left( x+u_{1}\right) -H\left( x+u_{2}\right) \right\vert \leq 
\frac{\varepsilon }{2\left( 1+\left\langle \mathbf{\varrho }\right\rangle _{%
\text{supp}\left( G\right) }\right) }
\end{equation*}%
for $\left\vert u_{1}-u_{2}\right\vert <\delta $. Then, for $\left\vert
u_{1}-u_{2}\right\vert <\delta $, $u_{1},u_{2}\in \boldsymbol{R}$ we have%
\begin{equation*}
\left\vert F_{H}\left( u_{1}\right) -F_{H}\left( u_{2}\right) \right\vert
=\left\vert \int\nolimits_{\boldsymbol{R}}\left( S_{1,u_{1}}H\left( x\right)
-S_{1,u_{2}}H\left( x\right) \right) \left\vert G\left( x\right) \right\vert 
\mathbf{\varrho }\left( x\right) dx\right\vert 
\end{equation*}%
\begin{equation*}
\leq \frac{\varepsilon }{2\left( 1+\left\langle \mathbf{\varrho }%
\right\rangle _{\text{supp}\left( G\right) }\right) }\int\nolimits_{%
\boldsymbol{R}}\left\vert G\left( x\right) \right\vert \mathbf{\varrho }%
\left( x\right) dx
\end{equation*}%
\begin{equation*}
\leq \frac{\varepsilon }{\left( 1+\left\langle \mathbf{\varrho }%
\right\rangle _{\text{supp}\left( G\right) }\right) }\left\langle \mathbf{%
\varrho }\right\rangle _{\text{supp}\left( G\right) }\left\Vert G\right\Vert
_{p^{\prime },\mathbf{\varrho }}<\varepsilon .
\end{equation*}%
Now, the conclusion follows for the class $C_{c}$. For the general case $%
f\in L_{p\left( \cdot \right) }$ there exists an $H\in C_{c}\left( 
\boldsymbol{R}\right) $ so that 
\begin{equation*}
\left\Vert f-H\right\Vert _{p\left( \cdot \right) }<\xi /(12\cdot 3^{\frac{2%
}{p}}\left[ \mathbf{\varrho }\right] _{p}^{1/p})
\end{equation*}%
for any $\xi >0.$ Then, for this $\xi $,%
\begin{equation*}
\left\vert F_{f}\left( u_{1}\right) -F_{f}\left( u_{2}\right) \right\vert
\leq \left\vert \int\nolimits_{\boldsymbol{R}}\mathsf{S}_{1,u_{1}}(f-H)%
\left( x\right) \left\vert G\left( x\right) \right\vert \mathbf{\varrho }%
\left( x\right) dx\right\vert +
\end{equation*}%
\begin{equation*}
+\left\vert \int\nolimits_{\boldsymbol{R}}\left( \mathsf{S}_{1,u_{1}}H\left(
x\right) -\mathsf{S}_{1,u_{2}}H\left( x\right) \right) \left\vert G\left(
x\right) \right\vert \mathbf{\varrho }\left( x\right) dx\right\vert 
\end{equation*}%
\begin{equation*}
+\left\vert \int\nolimits_{\boldsymbol{R}}\mathsf{S}_{1,u_{2}}(H-f)\left(
x\right) \left\vert G\left( x\right) \right\vert \mathbf{\varrho }\left(
x\right) dx\right\vert 
\end{equation*}%
\begin{equation*}
\leq \left\Vert \mathsf{S}_{1,u_{1}}\left( f-H\right) \right\Vert _{p,%
\mathbf{\varrho }}+\left\vert \int\nolimits_{\boldsymbol{R}}\mathsf{S}%
_{1}\left( H\left( x+u_{1}\right) -H\left( x+u_{2}\right) \right) \left\vert
G\left( x\right) \right\vert \mathbf{\varrho }\left( x\right) dx\right\vert 
\end{equation*}%
\begin{equation*}
+\left\Vert \mathsf{S}_{1,u_{2}}\left( f-H\right) \right\Vert _{p,\mathbf{%
\varrho }}
\end{equation*}%
\begin{equation*}
\leq 6\cdot 3^{\frac{2}{p}}\left[ \mathbf{\varrho }\right]
_{p}^{1/p}\left\Vert f-H\right\Vert _{p,\mathbf{\varrho }}+\xi /2\leq \xi
/2+\xi /2=\xi 
\end{equation*}%
As a result $F_{f}$ is bounded, uniformly continuous function defined on $%
\boldsymbol{R}.$

(b) Let $0\leq f,g\in L^{p}\left( \mathbf{\varrho }dx\right) $. If $%
\left\Vert f\right\Vert _{p,\varrho }=\left\Vert g\right\Vert _{p,\varrho
}=0 $, then, results (\ref{zz1}) is obvious. So we assume that $\left\Vert
f\right\Vert _{p,\varrho },\left\Vert g\right\Vert _{p,\varrho }\in \left(
0,+\infty \right) $. Then, 
\begin{align*}
\left\Vert F_{f}\right\Vert _{\mathcal{C}\left( \mathbf{R}\right) }& \leq
c_{1}\left\Vert F_{g}\right\Vert _{\mathcal{C}\left( \mathbf{R}\right)
}=c_{1}\left\Vert \int\nolimits_{\mathbf{R}}\mathsf{S}_{1,u}(g)\left(
x\right) \left\vert G\left( x\right) \right\vert \mathbf{\varrho }\left(
x\right) dx\right\Vert _{\mathcal{C}\left( \mathbf{R}\right) } \\
& =c_{1}\sup_{u\in \mathbf{R}}\int\nolimits_{\mathbf{R}}\mathsf{S}%
_{1,u}(g)\left( x\right) \left\vert G\left( x\right) \right\vert \mathbf{%
\varrho }\left( x\right) dx \\
& =c_{1}\sup_{u\in \mathbf{R}}\left\Vert \mathsf{S}_{1,u}(g)\right\Vert _{p,%
\mathbf{\varrho }}\left\Vert G\right\Vert _{p^{\prime },\mathbf{\varrho }%
}\leq 3c_{1}3^{\frac{2}{p}}\left[ \mathbf{\varrho }\right]
_{p}^{1/p}\left\Vert g\right\Vert _{p,\varrho }.
\end{align*}

On the other hand, for any $\varepsilon \in (0,\left\Vert f\right\Vert _{p,%
\mathbf{\varrho }}]$ we can choose $\bar{G}_{\varepsilon }\in L^{p^{\prime
}}\left( \mathbf{\varrho }dx\right) \cap S_{c}$ with 
\begin{equation*}
\int\nolimits_{\mathbf{R}}h\left( x\right) \left\vert \bar{G}_{\varepsilon
}\left( x\right) \right\vert \mathbf{\varrho }\left( x\right) dx\geq
\left\Vert h\right\Vert _{p,\mathbf{\varrho }}-\varepsilon \text{,\qquad }%
\left\Vert \bar{G}_{\varepsilon }\right\Vert _{p^{\prime },\mathbf{\varrho }%
}\leq 1\text{,}
\end{equation*}%
one can find%
\begin{eqnarray*}
\left\Vert F_{f}\right\Vert _{\mathcal{C}\left( \mathbf{R}\right) } &\geq
&\left\vert F_{f}\left( 0\right) \right\vert \geq \int\nolimits_{\mathbf{R}}%
\mathsf{S}_{1,0}f\left( x\right) \left\vert G\left( x\right) \right\vert 
\mathbf{\varrho }\left( x\right) dx \\
&=&\mathsf{S}_{1,0}\left( \int\nolimits_{\mathbf{R}}f\left( x\right)
\left\vert G\left( x\right) \right\vert \mathbf{\varrho }\left( x\right)
dx\right) \geq \mathsf{S}_{1,0}\left( \left\Vert f\right\Vert _{p,\mathbf{%
\varrho }}-\varepsilon \right) \\
&=&\left\Vert f\right\Vert _{p,\mathbf{\varrho }}-\varepsilon .
\end{eqnarray*}%
In the last inequality we take as $\varepsilon \rightarrow 0+$ and obtain $%
\left\Vert F_{f}\right\Vert _{\mathcal{C}\left( \mathbf{R}\right) }\geq
\left\Vert f\right\Vert _{p,\mathbf{\varrho }}$. Combining these
inequalities we get%
\begin{equation*}
\left\Vert f\right\Vert _{p,\mathbf{\varrho }}\leq \left\Vert
F_{f}\right\Vert _{\mathcal{C}\left( \mathbf{R}\right) }\leq c_{1}\left\Vert
F_{g}\right\Vert _{\mathcal{C}\left( \mathbf{R}\right) }\leq 3c_{1}3^{\frac{2%
}{p}}\left[ \mathbf{\varrho }\right] _{p}^{1/p}\left\Vert g\right\Vert
_{p,\varrho }.
\end{equation*}

In the general case $f,g\in L^{p}\left( \mathbf{\varrho }dx\right) $ we get%
\begin{equation}
\left\Vert f\right\Vert _{p,\varrho }\leq 6c_{1}3^{\frac{2}{p}}\left[ 
\mathbf{\varrho }\right] _{p}^{1/p}\left\Vert g\right\Vert _{p,\varrho }.
\label{sn1}
\end{equation}%
Then (\ref{zz1}) holds.
\end{proof}

\subsection{Averaging Operator and Mollifier}

\begin{definition}
\label{ddd}Let $B\subseteq \mathbf{R}$ be an open set, $\phi \in L_{1}\left(
B\right) $ and $\int\nolimits_{B}\phi \left( t\right) dt=1$. For each $t>0$
we define $\phi _{t}\left( x\right) =\frac{1}{t}\phi \left( \frac{x}{t}%
\right) $. Sequence $\left\{ \phi _{t}\right\} $ will be called approximate
identity. A function%
\begin{equation*}
\tilde{\phi}\left( x\right) =\sup\limits_{\left\vert y\right\vert \geq
\left\vert x\right\vert }\left\vert \phi \left( y\right) \right\vert 
\end{equation*}%
will be called radial majorant of $\phi .$ If $\tilde{\phi}\in L_{1}\left(
B\right) $, then, sequence $\left\{ \phi _{t}\right\} $ will be called
potential-type approximate identity.
\end{definition}

\begin{definition}
(a) Let $U\subset \mathbb{R}$ be \ a measurable set and%
\begin{equation*}
A_{U}f:=\frac{1}{\left\vert U\right\vert }\int\limits_{U}\left\vert f\left(
t\right) \right\vert dt.
\end{equation*}

(b) For a family $Q$ of open sets $U\subset \mathbb{R}$ we define averaging
operator by 
\begin{equation*}
T_{Q}:L_{loc}^{1}\rightarrow L^{0},
\end{equation*}%
\begin{equation*}
T_{Q}f\left( x\right) :=\sum_{U\in Q}\chi _{U}\left( x\right) A_{U}f,\quad
x\in \mathbb{R},
\end{equation*}%
where $L^{0}$ is the set of measurable functions on $\mathbb{R}$.

(c) For a measurable set $A\subset \mathbf{R}$, symbol $\left\vert
A\right\vert $ will represent the Lebesgue measure of $A$.
\end{definition}

Using Proposition 4.33 of \cite{uf13} and Part 5.2 of \cite[p.150]{dhhr11}
we have the following Theorem \ref{Aver}.

\begin{theorem}
\label{Aver}Suppose that $1\leq p<\infty $, $\mathbf{\varrho \in }A_{\mathbf{%
p}}$ and $f\in L^{p}\left( \mathbf{\varrho }dx\right) $. If $Q$ is $1$%
-finite family of open bounded subsets of $\mathbb{R}$, then, the averaging
operator $T_{Q}$ is uniformly bounded in $L^{p}\left( \mathbf{\varrho }%
dx\right) $, namely,%
\begin{equation*}
\left\Vert T_{Q\text{ }}f\right\Vert _{p,\mathbf{\varrho }}\leq \left[ 
\mathbf{\varrho }\right] _{p}^{1/p}\left\Vert f\right\Vert _{p,\mathbf{%
\varrho }}
\end{equation*}%
holds.
\end{theorem}

\begin{theorem}
\label{bt}Suppose that $1\leq p<\infty $, $\mathbf{\varrho \in }A_{\mathbf{p}%
}$, $f\in L^{p}\left( \mathbf{\varrho }dx\right) $, $\phi $ is a
potential-type approximate identity. Then, for any $t>0$,%
\begin{equation*}
\left\Vert f\ast \phi _{t}\right\Vert _{p,\mathbf{\varrho }}\leq 2\left\Vert 
\tilde{\phi}\right\Vert _{1}\mathbb{C}_{1}\left\Vert f\right\Vert _{p,%
\mathbf{\varrho }}
\end{equation*}%
and%
\begin{equation*}
\underset{t\rightarrow 0}{\lim }\left\Vert f\ast \phi _{t}-f\right\Vert _{p,%
\mathbf{\varrho }}=0
\end{equation*}%
hold.
\end{theorem}

\begin{proof}[\textbf{Proof of Theorem \protect\ref{bt}}]
We can use transference result. Since \cite[Theorem 12]{Ak17}%
\begin{equation*}
F_{f\ast \phi _{t}}=(F_{f})\ast \phi _{t}
\end{equation*}%
we find%
\begin{equation*}
\left\Vert f\ast \phi _{t}\right\Vert _{p,\mathbf{\varrho }}\leq \left\Vert
F_{f\ast \phi _{t}}\right\Vert _{\mathcal{C}\left( \mathbf{R}\right)
}=\left\Vert (F_{f})\ast \phi _{t}\right\Vert _{\mathcal{C}\left( \mathbf{R}%
\right) }\leq
\end{equation*}%
\begin{equation*}
\leq \left\Vert \tilde{\phi}\right\Vert _{1}\left\Vert F_{f}\right\Vert _{%
\mathcal{C}\left( \mathbf{R}\right) }\leq \mathbb{C}_{1}\left\Vert \tilde{%
\phi}\right\Vert _{1}\left\Vert f\right\Vert _{p,\mathbf{\varrho }}.
\end{equation*}
\end{proof}

\section{Exponential Approximation}

\begin{definition}
Let $X:=L^{p}\left( \mathbf{R}\right) $ or $L^{p}\left( \mathbf{\varrho }%
dx\right) $ or $\mathcal{C}\left( \mathbf{R}\right) $.

(i) We define $\mathcal{G}_{\sigma }\left( X\right) $ as the class of entire
function of exponential type $\sigma $ that belonging to $X$ . The quantity%
\begin{equation}
A_{\sigma }(f)_{X}:=\inf\limits_{g}\{\Vert f-g\Vert _{X}:g\in \mathcal{G}%
_{\sigma }\left( X\right) \}  \label{fd}
\end{equation}%
is called the deviation of the function $f\in X$ from $\mathcal{G}_{\sigma
}\left( X\right) $.

(ii) Let $W_{X}^{r}$, $r\in \mathbb{N}$, be the class of functions $f\in X$
such that derivatives $f^{\left( k\right) }$ exist for $k=1,...,r-1$, $%
f^{\left( r-1\right) }$ absolutely continuous and $f^{\left( r\right) }\in X$%
.

(iii) Define $\mathcal{G}_{\sigma }\left( p\right) :=\mathcal{G}_{\sigma
}\left( L^{p}\left( \mathbf{R}\right) \right) $, $\mathcal{G}_{\sigma
}\left( p\mathbf{,\varrho }\right) :=\mathcal{G}_{\sigma }\left( L^{p}\left( 
\mathbf{\varrho }dx\right) \right) $, $\mathcal{G}_{\sigma }\left( \infty
\right) :=\mathcal{G}_{\sigma }\left( \mathcal{C}\left( \mathbf{R}\right)
\right) $, $A_{\sigma }(f)_{p}:=A_{\sigma }(f)_{L^{p}\left( \mathbf{R}%
\right) }$, $A_{\sigma }(f)_{p\mathbf{,\varrho }}:=A_{\sigma
}(f)_{L^{p}\left( \mathbf{\varrho }dx\right) }$, $A_{\sigma }(f)_{\infty
}:=A_{\sigma }(f)_{\mathcal{C}\left( \mathbf{R}\right) }$, $%
W_{p}^{r}:=W_{L^{p}\left( \mathbf{R}\right) }^{r}$, $W_{p,\mathbf{\varrho }%
}^{r}:=W_{L^{p}\left( \mathbf{\varrho }dx\right) }^{r}$ and $W_{\infty
}^{r}:=W_{\mathcal{C}\left( \mathbf{R}\right) }^{r}$.
\end{definition}

In the following result of C. Bardaro, P.L. Butzer, R.L. Stens and G. Vinti,
exponential approximation result of the dela Val\`{e}e Poussin operator in $%
L^{p}\left( \mathbf{R}\right) $ $1\leq p\leq \infty $ was proved.

\begin{theorem}
\label{rmrk}\cite{bbsv06} Let $\sigma >0$, $1\leq p\leq \infty $, $f\in
L^{p}\left( \mathbf{R}\right) $,%
\begin{equation*}
\vartheta \left( x\right) :=\frac{2}{\pi }\frac{\sin \left( x/2\right) \sin
(3x/2)}{x^{2}}
\end{equation*}%
and%
\begin{equation*}
J\left( f,\sigma \right) =\sigma \int\nolimits_{\mathbf{R}}f\left(
x-u\right) \vartheta \left( \sigma u\right) du
\end{equation*}%
be the dela Val\`{e}e Poussin operator (\cite[definition given in (5.3)]%
{bbsv06}). It is known (see (5.4)-(5.5) of \cite{bbsv06}) that, if $f\in
L^{p}\left( \mathbf{R}\right) $, $1\leq p\leq \infty $, then,

(i) $J\left( f,\sigma \right) \in \mathcal{G}_{2\sigma }\left( p\right) $,

(ii) $J\left( g_{\sigma },\sigma \right) =g_{\sigma }$ for any $g_{\sigma
}\in \mathcal{G}_{\sigma }\left( p\right) $,

(iii) $\Vert J\left( f,\sigma \right) \Vert _{L_{p}\left( \mathbf{R}\right)
}\leq \frac{3}{2}\Vert f\Vert _{L_{p}\left( \mathbf{R}\right) }$,

(iv) $\left( J\left( f,\sigma \right) \right) ^{\left( r\right) }=J\left(
f^{\left( r\right) },\sigma \right) $ for any $r\in \mathbb{N}$ and $f\in
W_{p}^{r}$,

(v) $\Vert J\left( f,\frac{\sigma }{2}\right) -f\Vert _{L_{p}\left( \mathbf{R%
}\right) }\rightarrow 0$ (as $\sigma \rightarrow \infty $) and hence%
\begin{equation*}
\Vert \left( J\left( f,\frac{\sigma }{2}\right) \right) ^{\left( k\right)
}-f^{\left( k\right) }\Vert _{L_{p}\left( \mathbf{R}\right) }\rightarrow 0%
\text{ as }\sigma \rightarrow \infty ,
\end{equation*}%
for $f\in W_{p}^{r}$ and $1\leq k\leq r$.
\end{theorem}

For $r\in \mathbb{N}$, we define $C^{r}\left( \mathbf{R}\right) $ consisting
of every member $f\in C(\mathbf{R})$ such that the derivative $f^{\left(
k\right) }$ exists and is continuous on $\mathbf{R}$ for $k=1,...,r$.

\subsection{Modulus of Smoothness and translated Steklov average}

As a corollary of Theorem \ref{tra} we have

\begin{theorem}
\label{stekRR}Suppose that $1\leq p<\infty $, $\mathbf{\varrho \in }A_{%
\mathbf{p}}$, $0<\lambda <\infty $ and $\tau \in \mathbf{R}$. Then,

(i) $F_{\mathbf{S}_{\lambda ,\tau }f}=\mathbf{S}_{\lambda ,\tau }F_{f}$ and

(ii) the family of operators $\{\mathbf{S}_{\lambda ,\tau }f\}$, defined by (%
\ref{steklR}), is uniformly bounded (in $\lambda $ and $\tau $) in $%
L^{p}\left( \mathbf{\varrho }dx\right) $, namely, 
\begin{equation*}
\left\Vert \mathbf{S}_{\lambda ,\tau }f\right\Vert _{p,\mathbf{\varrho }%
}\leq \mathbb{C}_{1}\left\Vert f\right\Vert _{p,\mathbf{\varrho }}
\end{equation*}%
holds.
\end{theorem}

Proof of this theorem is a consequence of the transference Theorem \ref{tra}%
. Known results are proved under more restricted condition on $\lambda $ and 
$\tau $ such as $1\leq \lambda <\infty $ and $\left\vert \tau \right\vert
\leq \pi /\lfloor \lambda ^{\rho }\rfloor $ for some $\rho .$

\begin{corollary}
\label{coroL}Let $1\leq p<\infty $, $\mathbf{\varrho \in }A_{\mathbf{p}}$, $%
0<\delta <\infty $, $f\in L^{p}\left( \mathbf{\varrho }dx\right) $. If $\tau
=\delta /2$ then,%
\begin{equation*}
\mathbf{S}_{\delta ,\delta /2}f\left( x\right) =\frac{1}{\delta }%
\int\nolimits_{0}^{\delta }f\left( x+t\right) dt=T_{\delta }f\left( x\right)
,
\end{equation*}%
\begin{equation*}
F_{T_{\delta }f}=T_{\delta }F_{f}\text{ \ \ \ and}
\end{equation*}%
\begin{equation*}
\left\Vert T_{\delta }f\right\Vert _{p,\mathbf{\varrho }}\leq \mathbb{C}%
_{1}\left\Vert f\right\Vert _{p,\mathbf{\varrho }}.
\end{equation*}
\end{corollary}

For $1\leq p<\infty ,$ $\mathbf{\varrho \in }A_{\mathbf{p}}$, $f\in
L^{p}\left( \mathbf{\varrho }dx\right) $, $0<\delta <\infty $, $r\in \mathbb{%
N}$, we define modulus of smoothness%
\begin{equation*}
\Omega _{r}(f,\delta )_{p,\mathbf{\varrho }}:=\Vert (I-T_{\delta
})^{r}f\Vert _{p,\mathbf{\varrho }}.
\end{equation*}%
From Transference Result%
\begin{equation*}
\left\Vert (I-T_{\delta })^{r}f\right\Vert _{p,\mathbf{\varrho }}\leq 2^{r}%
\mathbb{C}_{1}\left\Vert f\right\Vert _{p,\mathbf{\varrho }}.
\end{equation*}

The following lemma is clear from definitions.

\begin{theorem}
\label{Fu}$1\leq p<\infty $, $\mathbf{\varrho }\in A_{\mathbf{p}}$. If $r\in 
\mathbb{N}$, and $f\in W_{p,\mathbf{\varrho }}^{r}$, then, $\frac{d^{k}}{%
du^{k}}F_{f}\left( u\right) $ exists and%
\begin{equation*}
\frac{d^{k}}{du^{k}}F_{f}\left( u\right) =F_{f^{\left( k\right) }}\left(
u\right) \text{ for }k\in \left\{ 1,...,r\right\} \text{, and }u\in \mathbf{R%
}.
\end{equation*}
\end{theorem}

We set $\lfloor \sigma \rfloor :=\max \left\{ n\in \mathbb{Z}:n\leq \sigma
\right\} $.

\begin{lemma}
\label{lemma1}Let $1\leq p<\infty $, $\mathbf{\varrho \in }A_{\mathbf{p}}$, $%
r\in \mathbb{N}$, and $0<\delta <\infty $. Then%
\begin{equation*}
\left\Vert \left( I-T_{\delta }\right) ^{r}f\right\Vert _{p,\mathbf{\varrho }%
}\leq \mathbb{C}_{1}\delta ^{r}\left\Vert f^{\left( r\right) }\right\Vert
_{p,\mathbf{\varrho }}\text{,}\quad f\in W_{p,\mathbf{\varrho }}^{r}
\end{equation*}%
hold.
\end{lemma}

\begin{proof}[\textbf{Proof of Lemma \protect\ref{lemma1}}]
We note that (see \cite{AG}) the following inequality 
\begin{equation}
\left\Vert \left( I-T_{\delta }\right) f\right\Vert _{p,\mathbf{\varrho }%
}\leq 2^{-1}\mathbb{C}_{1}\delta \left\Vert f^{\text{ }\prime }\right\Vert
_{p,\mathbf{\varrho }}\text{,}\quad \delta >0  \label{eqn6*}
\end{equation}%
holds for $f\in L^{p}\left( \mathbf{\varrho }dx\right) $. Then%
\begin{equation*}
\Omega _{r}\left( f,\delta \right) _{p,\mathbf{\varrho }}=\left\Vert \left(
I-T_{\delta }\right) ^{r}f\right\Vert _{p,\mathbf{\varrho }}\leq ...\leq
2^{-r}\mathbb{C}_{1}^{r}\delta ^{r}\left\Vert f^{(r)}\right\Vert _{p,\mathbf{%
\varrho }}\text{, \ \ }\delta >0
\end{equation*}%
for $f\in W_{p,\mathbf{\varrho }}^{r}$.
\end{proof}

\subsection{K-functional}

\begin{definition}
Let $X:=L^{p}\left( \mathbf{R}\right) $ or $L^{p}\left( \mathbf{\varrho }%
dx\right) $ or $\mathcal{C}\left( \mathbf{R}\right) $.

(i) Let $W_{X}^{r}$, $r\in \mathbb{N}$, be the class of functions $f\in X$
such that derivatives $f^{\left( k\right) }$ exist for $k=1,...,r-1$, $%
f^{\left( r-1\right) }$ absolutely continuous and $f^{\left( r\right) }\in X$%
. In particular, we set $W_{p}^{r}:=W_{L^{p}\left( \mathbf{R}\right) }^{r}$, 
$W_{p,\mathbf{\varrho }}^{r}:=W_{L^{p}\left( \mathbf{\varrho }dx\right)
}^{r} $ and $W_{\infty }^{r}:=W_{\mathcal{C}\left( \mathbf{R}\right) }^{r}$.

(ii) We define Peetre's \textit{K}-functional for the pair $X$ and $%
W_{X}^{r} $ as follows:%
\begin{equation*}
K_{r}\left( f,\delta ,X\right) :=\inf\limits_{g\in W_{X}^{r}}\left\{
\left\Vert f-g\right\Vert _{X}+\delta ^{r}\left\Vert g^{\left( r\right)
}\right\Vert _{X}\right\} \text{,\quad }\delta >0.
\end{equation*}

We will use notation $K_{r}\left( f,\delta ,p,\mathbf{\varrho }\right)
:=K_{r}\left( f,\delta ,L^{p}\left( \mathbf{\varrho }dx\right) \right) $ for 
$r\in \mathbb{N}$, $1\leq p<\infty $, $\mathbf{\varrho \in }A_{\mathbf{p}}$, 
$\delta >0$ and $f\in L^{p}\left( \mathbf{\varrho }dx\right) $. Also set $%
K_{r}\left( f,\delta ,C\right) :=K_{r}\left( f,\delta ,\mathcal{C}\left( 
\mathbf{R}\right) \right) $ for $r\in \mathbb{N}$, $\delta >0$ and $f\in 
\mathcal{C}\left( \mathbf{R}\right) .$
\end{definition}

As a corollary of Transference Result we can obtain the following lemma.

\begin{lemma}
\label{bukun}Let $0<h\leq \delta <\infty $, $1\leq p<\infty $, $\mathbf{%
\varrho \in }A_{\mathbf{p}}$, and $f\in L^{p}\left( \mathbf{\varrho }%
dx\right) $. Then%
\begin{equation}
F_{\left( I-T_{h}\right) f}=\left( I-T_{h}\right) F_{f}\quad \text{and}
\label{buk}
\end{equation}%
\begin{equation}
\left\Vert \left( I-T_{h}\right) f\right\Vert _{p,\mathbf{\varrho }}\leq 72%
\mathbb{C}_{1}\left\Vert \left( I-T_{\delta }\right) f\right\Vert _{p,%
\mathbf{\varrho }}  \label{bukunn}
\end{equation}%
holds.
\end{lemma}

\begin{proof}[\textbf{Proof of Lemma \protect\ref{bukun}}]
Property (\ref{buk}) follows from definitions of $F_{f}$ and $T_{h}$. On the
other hand, if $0<h\leq \delta <\infty $ and $g\in C\left( \mathbf{R}\right)
,$ then inequality%
\begin{equation}
\left\Vert \left( I-T_{h}\right) g\right\Vert _{C\left( \mathbf{R}\right)
}\leq 72\left\Vert \left( I-T_{\delta }\right) g\right\Vert _{C\left( 
\mathbf{R}\right) }  \label{eq1}
\end{equation}%
was proved in Proposition 2.6 of \cite{Akg45}. Now Transference Result, (\ref%
{eq1}) and (\ref{buk}) give the result (\ref{bukunn}).
\end{proof}

\begin{theorem}
\label{teo1}Let $1\leq p<\infty $, $\mathbf{\varrho \in }A_{\mathbf{p}}$ and 
$f\in L^{p}\left( \mathbf{\varrho }dx\right) $. Then,%
\begin{equation*}
\frac{1}{2^{r}\mathbb{C}_{1}}\leq \frac{K_{r}\left( f,\delta ,p,\mathbf{%
\varrho }\right) }{\Omega _{r}(f,\delta )_{p,\mathbf{\varrho }}}\leq \left\{
(2r)^{r}+2^{r}(34)^{r}\right\} \mathbb{C}_{1}.
\end{equation*}
\end{theorem}

\begin{proof}[\textbf{Proof of Theorem \protect\ref{teo1}}]
For any $g\in W_{p,\mathbf{\varrho }}^{r}$ we have $F_{g}\in W_{\infty }^{r}$%
, and $F_{\left( I-T_{\delta }\right) ^{r}f}=\left( I-T_{\delta }\right)
^{r}F_{f}$. Then using Theorem 2.5 of \cite{Akg45} and Transference Result%
\begin{equation*}
\left\Vert \left( I-T_{\delta }\right) ^{r}f\right\Vert _{p,\mathbf{\varrho }%
}\leq \left\Vert F_{\left( I-T_{\delta }\right) ^{r}f}\right\Vert _{\mathcal{%
C}\left( \mathbf{R}\right) }=\left\Vert \left( I-T_{\delta }\right)
^{r}F_{f}\right\Vert _{\mathcal{C}\left( \mathbf{R}\right) }
\end{equation*}%
\begin{equation*}
\leq 2^{r}K_{r}\left( F_{f},\delta ,C\right) \leq 2^{r}\left\{ \left\Vert
F_{f}-F_{g}\right\Vert _{\mathcal{C}\left( \mathbf{R}\right) }+\delta
^{r}\left\Vert \frac{d^{r}}{du^{r}}F_{g}\right\Vert _{\mathcal{C}\left( 
\mathbf{R}\right) }\right\}
\end{equation*}%
\begin{equation*}
\leq 2^{r}\left\{ \left\Vert F_{f-g}\right\Vert _{\mathcal{C}\left( \mathbf{R%
}\right) }+\delta ^{r}\left\Vert F_{g^{\left( r\right) }}\right\Vert _{%
\mathcal{C}\left( \mathbf{R}\right) }\right\}
\end{equation*}%
\begin{equation}
\leq 2^{r}\mathbb{C}_{1}\left\{ \left\Vert f-g\right\Vert _{p,\mathbf{%
\varrho }}+\delta ^{r}\left\Vert g^{\left( r\right) }\right\Vert _{p,\mathbf{%
\varrho }}\right\} .  \label{sss}
\end{equation}%
Now, taking infimum (\ref{sss}) and considering definition of \textit{K}%
-functional one gets%
\begin{equation*}
\left\Vert \left( I-T_{\delta }\right) ^{r}f\right\Vert _{p,\mathbf{\varrho }%
}\leq 2^{r}\mathbb{C}_{1}K_{r}\left( f,\delta ,p,\mathbf{\varrho }\right)
_{p,\mathbf{\varrho }}.
\end{equation*}%
Now we consider the opposite direction of the last inequality. For%
\begin{equation*}
g\left( \cdot \right) =\sum\limits_{l=1}^{r}\left( -1\right) ^{l-1}\binom{r}{%
l}T_{\delta }^{2rl}f\left( \cdot \right) ,
\end{equation*}%
using Theorem 2.5 of \cite{Akg45} and Transference Result we have%
\begin{equation*}
K_{r}\left( f,\delta ,p,\mathbf{\varrho }\right) _{p,\mathbf{\varrho }}\leq
\left\Vert f-g\right\Vert _{p,\mathbf{\varrho }}+\delta ^{r}\left\Vert \frac{%
d^{r}}{dx^{r}}g\right\Vert _{p,\mathbf{\varrho }}
\end{equation*}%
\begin{equation*}
\leq \left\Vert F_{f-g}\right\Vert _{\mathcal{C}\left( \mathbf{R}\right)
}+\delta ^{r}\left\Vert F_{g^{\left( r\right) }}\right\Vert _{\mathcal{C}%
\left( \mathbf{R}\right) }
\end{equation*}%
\begin{equation*}
=\left\Vert F_{f}-F_{g}\right\Vert _{\mathcal{C}\left( \mathbf{R}\right)
}+\delta ^{r}\left\Vert \frac{d^{r}}{du^{r}}F_{g}\right\Vert _{\mathcal{C}%
\left( \mathbf{R}\right) }
\end{equation*}%
\begin{equation*}
\leq \left\Vert \left( I-T_{\delta }^{2r}\right) ^{r}F_{f}\right\Vert _{%
\mathcal{C}\left( \mathbf{R}\right) }+\delta ^{r}\left\Vert \frac{d^{r}}{%
du^{r}}\sum\limits_{l=1}^{r}\left( -1\right) ^{l-1}\binom{r}{l}T_{\delta
}^{2rl}F_{f}\right\Vert _{\mathcal{C}\left( \mathbf{R}\right) }
\end{equation*}%
\begin{equation*}
=\left\Vert \left( I-T_{\delta }^{2r}\right) ^{r}F_{f}\right\Vert _{\mathcal{%
C}\left( \mathbf{R}\right) }+\sum\limits_{l=1}^{r}\left\vert \binom{r}{l}%
\right\vert \delta ^{r}\left\Vert \frac{d^{r}}{du^{r}}T_{\delta
}^{2rl}F_{f}\right\Vert _{\mathcal{C}\left( \mathbf{R}\right) }
\end{equation*}%
\begin{equation*}
\leq (2r)^{r}\left\Vert \left( I-T_{\delta }\right) ^{r}F_{f}\right\Vert _{%
\mathcal{C}\left( \mathbf{R}\right) }+2^{r}(34)^{r}\left\Vert \left(
I-T_{\delta }\right) ^{r}F_{f}\right\Vert _{\mathcal{C}\left( \mathbf{R}%
\right) }
\end{equation*}%
\begin{equation*}
=\left[ (2r)^{r}+2^{r}(34)^{r}\right] \left\Vert F_{\left( I-T_{\delta
}\right) ^{r}f}\right\Vert _{\mathcal{C}\left( \mathbf{R}\right) }
\end{equation*}%
\begin{equation*}
\leq \left\{ (2r)^{r}+2^{r}(34)^{r}\right\} \mathbb{C}_{1}\left\Vert \left(
I-T_{\delta }\right) ^{r}f\right\Vert _{p,\mathbf{\varrho }}.
\end{equation*}
\end{proof}

\begin{theorem}
\label{rem1}For $p\in \lbrack 1,\infty )$, $\mathbf{\varrho \in }A_{\mathbf{p%
}}$, $f,g\in L_{p}\left( \mathbf{R},\mathbf{\varrho }\right) $ and $\delta
>0,$ the following properties are hold:

\begin{enumerate}
\item $\Omega _{r}\left( f,\delta \right) _{p,\mathbf{\varrho }}$ is
non-negative; non-decreasing function of $\delta $;

\item $\Omega _{r}(f,\delta )_{p,\mathbf{\varrho }}$ is sub-additive of $f$;

\item 
\begin{equation}
\lim\limits_{\delta \rightarrow 0}K_{r}\left( f,\delta ,p,\mathbf{\varrho }%
\right) =0.  \label{Kftendzero}
\end{equation}
As a result%
\begin{equation}
\lim\limits_{\delta \rightarrow 0}\Omega _{r}(f,\delta )_{p,\mathbf{\varrho }%
}=0.  \label{omegaTendZero}
\end{equation}
\end{enumerate}
\end{theorem}

\begin{proof}[\textbf{Proof of Theorem \protect\ref{rem1}}]
Properties (1) and (2) are clear from definition. Since%
\begin{equation*}
\lim\limits_{\delta \rightarrow 0}K_{r}\left( f,\delta ,p,\mathbf{\varrho }%
\right) =0,
\end{equation*}%
we have, from Theorem \ref{teo1}, that (\ref{onL1}) holds.
\end{proof}

\subsection{Jackson type inequality}

\begin{theorem}
\label{jak}Let $p\in \lbrack 1,\infty )$, $\mathbf{\varrho \in }A_{\mathbf{p}%
}$, $r\in \mathbb{N}$, $\sigma >0$ and $f\in L^{p}\left( \mathbf{\varrho }%
dx\right) $. Then,%
\begin{equation}
A_{\sigma }\left( f\right) _{p,\mathbf{\varrho }}\leq 25\pi 8^{r-1}\mathbb{C}%
_{2}\mathbb{C}_{1}\left\Vert \left( I-T_{1/\sigma }\right) ^{r}f\right\Vert
_{p,\mathbf{\varrho }}.  \label{JJ}
\end{equation}
\end{theorem}

\begin{proof}[\textbf{Proof of Theorem \protect\ref{jak}}]
First we obtain%
\begin{equation}
A_{2\sigma }\left( f\right) _{p,\mathbf{\varrho }}\leq 25\pi 8^{r-1}\mathbb{C%
}_{2}\mathbb{C}_{1}\left\Vert \left( I-T_{1/\left( 2\sigma \right) }\right)
^{r}f\right\Vert _{p,\mathbf{\varrho }}  \label{JE}
\end{equation}%
and (\ref{JJ}) follows from (\ref{JE}). Let us take $g_{\sigma }\in \mathcal{%
G}_{\sigma }\left( \infty \right) $ with $\left\Vert F_{f}-g_{\sigma
}\right\Vert _{\mathcal{C}(\mathbf{R})}=A_{\sigma }\left( F_{f}\right) _{%
\mathcal{C}(\mathbf{R})}$. Using%
\begin{equation*}
V_{\sigma }F_{f}=F_{V_{\sigma }f}
\end{equation*}%
and $V_{\sigma }g_{\sigma }=g_{\sigma }$ we get%
\begin{equation*}
A_{2\sigma }\left( f\right) _{p,\mathbf{\varrho }}\leq \left\Vert
f-V_{\sigma }f\right\Vert _{p,\mathbf{\varrho }}\leq \left\Vert
F_{f-V_{\sigma }f}\right\Vert _{\mathcal{C}(\mathbf{R})}=\left\Vert
F_{f}-V_{\sigma }F_{f}\right\Vert _{\mathcal{C}(\mathbf{R})}
\end{equation*}%
\begin{equation*}
\leq \left\Vert F_{f}-g_{\sigma }+g_{\sigma }-V_{\sigma }F_{f}\right\Vert _{%
\mathcal{C}(\mathbf{R})}=\left\Vert F_{f}-g_{\sigma }+V_{\sigma }g_{\sigma
}-V_{\sigma }F_{f}\right\Vert _{\mathcal{C}(\mathbf{R})}
\end{equation*}%
\begin{equation*}
\leq A_{\sigma }\left( F_{f}\right) _{\mathcal{C}(\mathbf{R})}+\frac{3}{2}%
A_{\sigma }\left( F_{f}\right) _{\mathcal{C}(\mathbf{R})}=\frac{5}{2}%
A_{\sigma }\left( F_{f}\right) _{\mathcal{C}(\mathbf{R})}.
\end{equation*}

For any $g\in W_{\infty }^{r}$%
\begin{equation*}
A_{\sigma }\left( u\right) _{\mathcal{C}(\mathbf{R})}\leq A_{\sigma }\left(
u-g\right) _{\mathcal{C}(\mathbf{R})}+A_{\sigma }\left( g\right) _{\mathcal{C%
}(\mathbf{R})}
\end{equation*}%
\begin{equation*}
\leq \left\Vert u-g\right\Vert _{\mathcal{C}(\mathbf{R})}+\frac{5\pi }{4}%
\frac{4^{r}}{\sigma ^{r}}\left\Vert \frac{d^{r}}{dx^{r}}g\right\Vert _{%
\mathcal{C}(\mathbf{R})}
\end{equation*}%
\begin{equation*}
\leq \frac{5\pi 4^{r}}{4}K_{r}\left( u,\sigma ^{-1},C\right) \leq \frac{5\pi
8^{r}}{4}K_{r}\left( u,\frac{1}{2\sigma },\mathcal{C}(\mathbf{R})\right) _{%
\mathcal{C}(\mathbf{R})}
\end{equation*}%
\begin{equation*}
\leq \frac{5\pi 8^{r}}{4}\mathbb{C}_{2}\left\Vert \left( I-T_{\left( 2\sigma
\right) ^{-1}}\right) ^{r}u\right\Vert _{\mathcal{C}(\mathbf{R})}.
\end{equation*}%
where $\mathbb{C}_{2}:=\mathbb{C}_{2}\left( 1\right) =1$ and $\mathbb{C}%
_{2}\left( r\right) =2^{r}\left( r^{r}+(34)^{r}\right) $ for $r>1$ (see
Theorem 2.5 of \cite{Akg45}). Therefore%
\begin{equation*}
A_{2\sigma }\left( f\right) _{p,\mathbf{\varrho }}\leq \frac{5}{2}A_{\sigma
}\left( F_{f}\right) _{\mathcal{C}(\mathbf{R})}\leq 25\pi 8^{r-1}\mathbb{C}%
_{2}\left\Vert \left( I-T_{\frac{1}{2\sigma }}\right) ^{r}F_{f}\right\Vert _{%
\mathcal{C}(\mathbf{R})}
\end{equation*}%
\begin{equation*}
=25\pi 8^{r-1}\mathbb{C}_{2}\left\Vert F_{\left( I-T_{1/\left( 2\sigma
\right) }\right) ^{r}f}\right\Vert _{\mathcal{C}(\mathbf{R})}\leq 25\pi
8^{r-1}\mathbb{C}_{2}\mathbb{C}_{1}\left\Vert \left( I-T_{1/\left( 2\sigma
\right) }\right) ^{r}f\right\Vert _{p,\mathbf{\varrho }}.
\end{equation*}
\end{proof}

\subsection{Inverse theorem}

\begin{theorem}
\label{Ters T}Let $p\in \lbrack 1,\infty )$, $\mathbf{\varrho \in }A_{%
\mathbf{p}}$, $r\in \mathbb{N}$, $\delta \in \left( 0,\infty \right) $ and $%
f\in L^{p}\left( \mathbf{\varrho }dx\right) $. Then,%
\begin{equation*}
\Omega _{r}\left( f,\delta \right) _{p,\mathbf{\varrho }}\leq \mathbb{C}%
_{3}\delta ^{r}\left( A_{0}\left( f\right) _{p,\mathbf{\varrho }%
}+\int\nolimits_{1/2}^{1/\delta }u^{r-1}A_{u/2}\left( f\right) _{p,\mathbf{%
\varrho }}du\right)
\end{equation*}%
holds with $\mathbb{C}_{3}$:=$\mathbb{C}_{1}\left( 1+3\mathbb{C}_{1}\right)
2^{r+1}\left( 1+2^{2r-1}\right) $.
\end{theorem}

\begin{proof}[\textbf{Proof of Theorem \protect\ref{Ters T}}]
\begin{equation*}
\Omega _{r}\left( f,\delta \right) _{p,\mathbf{\varrho }}=\left\Vert \left(
I-T_{\delta }\right) ^{r}f\right\Vert _{p,\mathbf{\varrho }}\leq \left\Vert
F_{\left( I-T_{\delta }\right) ^{r}f}\right\Vert _{\mathcal{C}(\mathbf{R}%
)}=\left\Vert \left( I-T_{\delta }\right) ^{r}F_{f}\right\Vert _{\mathcal{C}(%
\mathbf{R})}
\end{equation*}%
\begin{equation*}
\leq 2^{r}\left( 1+2^{2r-1}\right) \delta ^{r}\left( A_{0}\left( \Xi
_{f}\right) _{\mathcal{C}(\mathbf{R})}+\int_{1/2}^{1/\delta
}u^{r-1}A_{u}\left( F_{f}\right) _{\mathcal{C}(\mathbf{R})}du\right)
\end{equation*}%
\begin{equation*}
\leq \mathbb{C}_{1}\left( 1+3\mathbb{C}_{1}\right) 2^{r}\left(
1+2^{2r-1}\right) \delta ^{r}\left( A_{0}\left( f\right) _{p,\mathbf{\varrho 
}}+\int_{1/2}^{1/\delta }u^{r-1}A_{u/2}\left( f\right) _{p,\mathbf{\varrho }%
}du\right)
\end{equation*}%
because%
\begin{equation*}
A_{2\sigma }\left( F_{f}\right) _{\mathcal{C}(\mathbf{R})}\leq \left\Vert
F_{f}-V_{\sigma }F_{f}\right\Vert _{\mathcal{C}(\mathbf{R})}=\left\Vert
F_{f-V_{\sigma }f}\right\Vert _{\mathcal{C}(\mathbf{R})}
\end{equation*}%
\begin{equation*}
\leq \mathbb{C}_{1}\left\Vert f-V_{\sigma }f\right\Vert _{p,\mathbf{\varrho }%
}=\mathbb{C}_{1}\left\Vert f-g_{\sigma }+g_{\sigma }-V_{\sigma }f\right\Vert
_{p,\mathbf{\varrho }}
\end{equation*}%
\begin{equation*}
\leq \mathbb{C}_{1}\left( \left\Vert f-g_{\sigma }\right\Vert _{p,\mathbf{%
\varrho }}+\left\Vert V_{\sigma }g_{\sigma }-V_{\sigma }f\right\Vert _{p,%
\mathbf{\varrho }}\right)
\end{equation*}%
\begin{equation*}
\leq \mathbb{C}_{1}\left( \left\Vert f-g_{\sigma }\right\Vert _{p,\mathbf{%
\varrho }}+3\mathbb{C}_{1}\left\Vert g_{\sigma }-f\right\Vert _{p,\mathbf{%
\varrho }}\right) =\mathbb{C}_{1}\left( 1+3\mathbb{C}_{1}\right) A_{\sigma
}\left( f\right) _{_{p,\mathbf{\varrho }}}.
\end{equation*}
\end{proof}

\subsection{Marchaud inequality}

\begin{theorem}
\label{Ters}Let $r,k\in \mathbb{N}$, $1\leq p<\infty $, $\mathbf{\varrho \in 
}A_{\mathbf{p}}$, $f\in L^{p}\left( \mathbf{\varrho }dx\right) $ and $t\in
\left( 0,1/2\right) $. Then,%
\begin{equation*}
\Omega _{r}\left( f,t\right) _{p,\mathbf{\varrho }}\leq \mathbb{C}%
_{4}t^{r}\int\nolimits_{t}^{1}\frac{\Omega _{r\text{+}k}\left( f,u\right)
_{p,\mathbf{\varrho }}}{u^{r+1}}du
\end{equation*}%
holds with $\mathbb{C}_{4}$:=$20\pi \mathbb{C}_{1}\left( 1+2^{2r-1}\right)
2_{\;}^{2r+3k}\mathbb{C}_{2}\left( r+k\right) $ where $\mathbb{C}_{2}\left(
1\right) :=36$, and $\mathbb{C}_{2}\left( r\right) :=2^{r}\left(
r^{r}+(34)^{r}\right) $ for $r>1$.
\end{theorem}

\begin{proof}[\textbf{Proof of Theorem \protect\ref{Ters}}]
Let $\sigma >0$ and $g_{\sigma }$ be an exponential type entire function of
degree $\leq \sigma $, belonging to $L^{p}\left( \mathbf{\varrho }dx\right) $%
, as best approximation of $f\in L^{p}\left( \mathbf{\varrho }dx\right) $.
Then%
\begin{equation*}
\Omega _{r}\left( f,t\right) _{p,\mathbf{\varrho }}=\left\Vert \left(
I-T_{t}\right) ^{r}f\right\Vert _{p,\mathbf{\varrho }}\leq \left\Vert
F_{\left( I-T_{t}\right) ^{r}f}\right\Vert _{\mathcal{C}(\mathbf{R}%
)}=\left\Vert \left( I-T_{t}\right) ^{r}F_{f}\right\Vert _{\mathcal{C}(%
\mathbf{R})}
\end{equation*}%
\begin{equation*}
\leq (\mathbb{C}_{4}\mathbb{C}_{1})t^{r}\int_{t}^{1}\frac{\left\Vert \left(
I-T_{t}\right) ^{r+k}F_{f}\right\Vert _{\mathcal{C}(\mathbf{R})}}{u^{r+1}}du
\end{equation*}%
\begin{equation*}
=(\mathbb{C}_{4}\mathbb{C}_{1})t^{r}\int_{t}^{1}\frac{\left\Vert F_{\left(
I-T_{t}\right) ^{r+k}f}\right\Vert _{\mathcal{C}(\mathbf{R})}}{u^{r+1}}du
\end{equation*}%
\begin{equation*}
\leq \mathbb{C}_{4}t^{r}\int_{t}^{1}\frac{\left\Vert \left( I\text{-}%
T_{t}\right) ^{r+k}f\right\Vert _{p,\mathbf{\varrho }}}{u^{r+1}}du
\end{equation*}%
\begin{equation*}
=\mathbb{C}_{4}t^{r}\int_{t}^{1}\frac{\Omega _{r+k}\left( f,t\right) _{p,%
\mathbf{\varrho }}}{u^{r+1}}du.
\end{equation*}
\end{proof}

\subsection{Inverse theorem for derivatives}

\begin{theorem}
\label{crv}Let $1\leq p<\infty $, $\mathbf{\varrho \in }A_{\mathbf{p}}$, $%
r\in \mathbb{N}$ and $f\in L^{p}\left( \mathbf{\varrho }dx\right) $. If%
\begin{equation*}
\sum\limits_{\nu =0}^{\infty }\nu ^{k-1}A_{\nu /2}\left( f\right) _{p,%
\mathbf{\varrho }}<\infty
\end{equation*}%
holds for some $k\in \mathbb{N}$, then $f^{\left( k\right) }\in L^{p}\left( 
\mathbf{\varrho }dx\right) $ and%
\begin{equation}
\Omega _{r}\left( f^{\left( k\right) },\frac{1}{\sigma }\right) _{p,\mathbf{%
\varrho }}\leq \mathbb{C}_{5}\left( \frac{1}{\sigma ^{r}}\sum\limits_{\nu
=0}^{\lfloor \sigma \rfloor }\left( \nu +1\right) ^{r+k-1}A_{\nu /2}\left(
f\right) _{p,\mathbf{\varrho }}+\sum\limits_{\nu =\lfloor \sigma \rfloor
+1}^{\infty }\nu ^{k-1}A_{\nu /2}\left( f\right) _{p,\mathbf{\varrho }%
}\right)  \label{Sinv}
\end{equation}%
with $\mathbb{C}_{5}=2^{2k+r+1}\mathbb{C}_{1}$.
\end{theorem}

\begin{proof}[\textbf{Proof of Theorem \protect\ref{crv}}]
Proof of (\ref{Sinv}) is similar to that of proof of Theorem \ref{Ters}. See
Theorem 2.15 of \cite{Akg45}.
\end{proof}


\begin{thebibliography}{99}
\bibitem{ascs} F. Abdullaev, A. Shidlich and S. Chaichenko, \textit{Direct
and inverse approximation theorems of functions in the Orlicz type spaces},
Math. Slovaca, 69 (2019), No:6, 1367-1380.

\bibitem{Ack} N. I. Ackhiezer, \textit{Lectures on theory of approximation},
Fizmatlit, Moscow, 1965; English transl. of 2nd ed. Frederick Ungar, New
York, 1956.

\bibitem{ra11u} R. Akg\"{u}n, \textit{Approximation of Functions of Weighted
Lebesgue and Smirnov Spaces}, Mathematica (Cluj), Tome 54 (77), No: Special
(2012), pp. 25-36.

\bibitem{AK1} R. Akg\"{u}n, \textit{Sharp Jackson and converse theorems of
trigonometric approximation in weighted Lebesgue spaces}, Proc. A. Razmadze
Math. Inst., 152 (2010), pp. 1-18.

\bibitem{eja} R. Akg\"{u}n, \textit{Polynomial approximation in weighted
Lebesgue spaces}, East J. Approx., vol: 17, no 3, (2011), pp. 253-266.

\bibitem{spbu} R. Akg\"{u}n, \textit{Realization and characterization of
modulus of smoothness in weighted Lebesgue spaces}, St. Petersburg Math. J.
26 (2015), 741-756.

\bibitem{raTjm} R. Akg\"{u}n, \textit{Gadjieva's conjecture, K-functionals
and some applications in Lebesgue spaces}, Turk. J. Math., 42 (2018), No:3,
1484-1503.

\bibitem{Akg45} R. Akg\"{u}n, Exponential approximation in variable exponent
Lebesgue spaces on the real line, arXiv:2109.02083v1 [math.FA],
https://doi.org/10.48550/arXiv.2109.02083

\bibitem{Ak17} R. Akg\"{u}n, \textit{Weighted norm inequalities in Lebesgue
spaces with Muckenhoupt weights and some applications to operators},
arXiv:1709.02928v4 [math.CA]. https://doi.org/10.48550/arXiv.1709.02928

\bibitem{AG} R. Akg\"{u}n; A. Ghorbanalizadeh, \textit{Approximation by
integral functions of finite degree in variable exponent Lebesgue spaces on
the real axis}, Turk. J. Math. 42, (2018), no. 4, 1887--1903.

\bibitem{vva} V. V. Arestov, \textit{On Jackson inequalities for
approximation in L}$_{2}$\textit{\ of periodic functions by trigonometric
polynomials and of functions on the line by entire functions}, In:
Approximation Theory: A Volume Dedicated to Borislaw Bojanov, Marin Drinov
Acad. Publ. House, Sofia (2004), pp. 1-19.

\bibitem{art} S. Artamonov, \textit{On some constructions of a non-periodic
modulus of smoothness related to the Riesz derivative}, Eurasian
Mathematical Journal. 2018. Vol. 9. No. 2. P. 11-21.

\bibitem{arsc} S. Artamonov, K. Runovski and H.-J. Schmeisser, \textit{%
Approximation by band-limited functions, generalized K-functionas and
generalized moduli of smoothness}, Analysis Math., 45 (2019), No: 1 , 1-24.

\bibitem{ar1} S. Artamonov, K. Runovski, H. J. Schmeisser, \textit{%
Approximation by families of generalized sampling series, realizations of
generalized K-functionals and generalized moduli of smoothness}, Journal of
Mathematical Analysis and Applications. 2020. Vol. 489. No. 1. P. 1-19.

\bibitem{ahak} A.H. Av\c{s}ar and H. Ko\c{c}, \textit{Jackson and Stechkin
type inequalities of trigonometric approximation in }$A_{p,q(.)}^{w,\theta }$%
, Turk J Math 42 (2018), No:6, 2979-2993.

\bibitem{babe} A. G. Babenko, \textit{Exact Jackson--Stechkin inequality in
the space L}$_{2}$\textit{(R}$^{m}$\textit{)}, in: Proc. of the Institute of
Mathematics and Mechanics, Ural Division of the Russian Academy of Sciences
[in Russian], No. 5 (1998), pp. 3-7.

\bibitem{bbsv06} C. Bardaro, P.L. Butzer, R.L. Stens and G. Vinti, \textit{%
Approximation error of the Whittaker cardinal series in terms of an averaged
modulus of smoothness covering discontinuous signals}, J. Math. Anal. Appl.
316 (2006), No: 1, 269-306.

\bibitem{bg03} E. Berkson and T. A. Gillespie, \textit{On restrictions of
multipliers in weighted setting}, Indian U. Math. J., 52 (2003), no:4,
927-961.

\bibitem{B12} S. N. Bernstein, \textit{On the best approximation of
continuous functions on the entire real axis with the use of entire
functions of given degree} (1912); in: Collected Works, Vol. 2 [in Russian],
Izd. Akad. Nauk SSSR, Moscow (1952), pp. 371-375.

\bibitem{B46} S.N. Bernstein, \textit{Sur la meilleure approximation sur
tout l'axe reel des fonctions continues par des fonctions entieres de degre
n. I}, C.R. (Doklady) Acad. Sci. URSS (N.S.) 51 (1946), 331-334.

\bibitem{RPB} R. P. Boas Jr., \textit{Entire functions}. Pure and Applied
Mathematics, vol. 5. New York, Academic Press, 1954. 12+276 pp.

\bibitem{BuRuSc06} Z. Burinska, K. Runovski, H. J. Schmeisser, On the
Approximation by Generalized Sampling Series in $L_{p}$-Metrics, STSIP 5,
59-87 (2006).

\bibitem{BuRuSc} Z. Burinska, K. Runovski, H. J. Schmeisser, On Quality of
Approximation by Families of Generalized Sampling Series, STSIP 8, 105-126
(2009).

\bibitem{BSS88} P. L. Butzer, W. Splettst\"{o}$\beta $er, and R. L. Stens, 
\textit{The sampling theorem and linear prediction in signal analysis},
Jahresber. Deutsch. Math.-Verein, 90 (1988), pp. 1-70.

\bibitem{uf13} D. V. Cruz-Uribe and A. Fiorenza, Variable Lebesgue Spaces,
Foundations and Harmonic Analysis, Birkh\"{a}user, 2013.

\bibitem{DDTi} F. Dai, Z. Ditzian, S. Yu. Tikhonov, Sharp Jackson
inequalities. J. Approx. Theory 151(1): 86-112 (2008).

\bibitem{devore} R. A. Devore, G. G. Lorentz, \textit{Constructive
Approximation}, Springer-Verlag, (1993).

\bibitem{dhhr11} L. Diening, P. Harjulehto, Peter H\"{a}st\"{o}, Michael R%
\r{u}\v{z}i\v{c}ka, Lebesgue and Sobolev Spaces with Variable Exponents,
Lecture Notes in Mathematics 2017, 2011.

\bibitem{DL} Z. Ditzian, D. S. Lubinsky, Jackson and smoothness theorems for
Freud weights, Constr. Approx. 13 1997, No: 1, 99-152.

\bibitem{Dit} Z. Ditzian and K. G. Ivanov, \textit{Strong converse
inequalities}, Journal D'analyse mathematique \textbf{61} (1993), 61-111.

\bibitem{DP} Z. Ditzian, A Prymak, \textit{Convexity, moduli of smoothness
and a Jackson-type inequality}, Acta Mathematica Hungarica 130 (3), 254-285.

\bibitem{DiRu} Z. Ditzian and K. V. Runovski, \textit{Averages and
K-functionals related to the Laplacian}, J. Approx. Theory, Volume 97, Issue
1, 1999, 113-139.

\bibitem{DQR03} D. P. Dryanov, M. A. Qazi, and Q. I. Rahman, \textit{Entire
functions of exponential type in Approximation Theory}, In:\ Constructive
Theory of Functions, Varna 2002 (B. Bojanov, Ed.), DARBA, Sofia, 2003, pp.
86-135.

\bibitem{Gad} \'{E}. A. Gadjieva, \textit{Investigation of the properties of
functions with quasimonotone Fourier coefficients in generalized
Nikolskii-Besov spaces}, PhD Dissertation, Tbilisi, 1986, (In Russian).

\bibitem{gaim1} G. Gaimnazarov, \textit{On the moduli of continuity of
fractional order for functions given on the entire real axis}, Dokl. Akad.
Nauk Tadzhik. SSR, 24, No. 3, 148-150 (1981).

\bibitem{GIvTi} D.V. Gorbachev, V.I. Ivanov, and S.Yu. Tikhonov, \textit{%
Sharp approximation theorems and Fourier inequalities in the Dunkl setting},
J. Approx. Theory, Volume 258, October 2020, 105462.

\bibitem{GIv} D.V. Gorbachev, V.I. Ivanov, \textit{Fractional smoothness in }%
$L_{p}$\textit{\ with Dunkl weight and its applications}, Math. Notes 106
(4) (2019) 537--561.

\bibitem{GorIva} D.V. Gorbachev, V.I. Ivanov, \textit{A sharp Jackson
inequality in }$\mathit{L}_{\mathit{p}}$\textit{(}$\mathit{R}^{\mathit{d}}$%
\textit{) with Dunkl weight}, Math. Notes 105 (5--6) (2019) 657--673.

\bibitem{GorIvTikhon} D.V. Gorbachev, V.I. Ivanov, S.Yu. Tikhonov, \textit{%
Positive }$L_{p}$\textit{-bounded Dunkl-type generalized translation
operator and its applications}, Constr. Approx. 49 (3) (2019) 555-605.

\bibitem{GK} A. Guven and V. Kokilashvili, \textit{On the means of Fourier
integrals and Bernstein inequality in the two-weighted setting}, Positivity
14 (2010), No: 1, 165-180.

\bibitem{GKstev} A. Guven and V. Kokilashvili, \textit{Improved inverse
theorems in weighted Lebesgue and Smirnov spaces}, Bulletin of the Belgian
Mathematical Society-Simon Stevin 14 (4), 681-692.

\bibitem{II3} I.I. Ibragimov, \textit{Teoriya priblizheniya tselymi
funktsiyami. (Russian) [The theory of approximation by entire functions]}
"Elm\textquotedblright , Baku, 1979, 468 pp.

\bibitem{Ja1} S.Z. Jafarov, \textit{Approximation by means of Fourier
trigonometric deries in weighted Lebesgue spaces}, Sarajevo J. Math.,Vol.13
(26), No.2, (2017), 217-226.

\bibitem{ja2} S. Z. Jafarov, \textit{On moduli of smoothness of functions in
Orlicz spaces}, Tbilisi Math. J. 12(3): 121-129 (2019).

\bibitem{koklSamko} V. Kokilashvili and S. Samko, \textit{Singular integrals
in weighted Lebesgue spaces with variable exponent}, Georgian Mathematical
Journal 10 (2003), No:1, 145-156.

\bibitem{KolTik} Y. Kolomoitsev, S. Tikhonov, \textit{Properties of moduli
of smoothness in} $L_{p}(\mathbb{R}^{d})$, J. Approx. Theory 257 (2020),
105423.

\bibitem{Ky1} N. X. Ky, \textit{On approximation by trigonometric
polynomials in }$L_{u}^{p}$\textit{-spaces}, Studia Sci. Math. Hungar. 28
(1993), no. 1-2, 183--188.

\bibitem{Ky2} N. X. Ky, \textit{Moduli of mean smoothness and approximation
with }$A_{p}$\textit{-weights}, Ann. Univ. Sci. Budapest. E\"{o}tv\"{o}s
Sect. Math. 40 (1997), 37--48 (1998).

\bibitem{maskh} H. N. Mhaskar, \textit{Introduction to the Theory of
Weighted Polynomial Approximation, }1997, Pages: 396, Series in
Approximations and Decompositions: Volume 7, World Scientific.

\bibitem{Bm72} B. Muckenhoupt, \textit{Weighted norm inequalities for the
Hardy maximal function}, Trans. Amer. Math. Soc. 165 (1972), 207-226.

\bibitem{FGN} F. G. Nasibov, \textit{Approximation in }$L_{2}$\textit{\ by
entire functions}.(Russian) Akad. Nauk Azerbaidzhan. SSR Dokl. 42 (1986),
No: 4, 3-6.

\bibitem{SMN51} S.M. Nikol'ski\u{\i}, \textit{Inequalities for entire
functions of finite degree and their application in the theory of
differentiable functions of several variables}, Collection of articles. To
the sixtieth birthday of academician Ivan Matveevich Vinogradov, Trudy Mat.
Inst. Steklov., 38, Acad. Sci. USSR, Moscow, 1951, 244-278.

\bibitem{SMNbook} S.M. Nikol'ski\u{\i}, \textit{Approximation of Functions
of Several Variables and Imbedding Theorems}, Die Grundlehren der
mathematischen Wissenshaften, Band 205, Springer-Verlag, New York,
Heidelberg, Berlin, 1975, viii + 420 pp.

\bibitem{LD} A.A. Ligun and V.G. Doronin, \textit{Exact constants in
Jackson-type inequalities for the }$L_{2}$\textit{-approximation on a
straight line}. Translation in Ukrainian Math. J. 61 (2009), No: 1, 112-120.

\bibitem{ponom} V. G. Ponomarenko, \textit{Fourier integrals and the best
approximation by entire functions}, Izv. Vyssh. Uchebn. Zaved., Ser. Mat.,
No. 3, 109-123 (1966).

\bibitem{Po} V. Yu. Popov, \textit{Best mean square approximations by entire
functions of exponential type}. (Russian) Izv. Vys\v{s}. Ucebn. Zaved.
Matematika. ; 121 (1972), No: 6, 65-73.

\bibitem{hjssic} H. J. Schmeisser, W. Sickel, \textit{Sampling theory and
function spaces}, Applied Mathematics Reviews: (Volume 1), 2000.

\bibitem{Ste1} A. I. Stepanets, \textit{Classes of functions defined on the
real line and their approximations by entire functions. I}, Ukr. Mat. Zh.,
42, No. 1,102-112 (1990); English translation: Ukr. Math. J., 42, No. 1,
93-102 (1990).

\bibitem{Ste2} A. I. Stepanets, \textit{Classes of functions defined on the
real axis and their approximations by entire functions. II,}\ Ukr. Mat. Zh.,
42, No. 2, 210-222 (1990); English translation: Ukr. Math. J., 42, No. 2,
186-197 (1990).

\bibitem{AFT} A.F. Timan, \textit{Theory of approximation of functions of a
real variable}. Translated from the Russian by J. Berry. English translation
edited and editorial preface by J. Cossar. International Series of
Monographs in Pure and Applied Mathematics, Vol. 34, The Macmillan Co., New
York: A Pergamon Press Book. 1963.

\bibitem{MFT} M. F. Timan, \textit{The approximation of functions defined on
the whole real axis by entire functions of exponential type}. Izv. Vyssh.
Uchebn. Zaved. Mat. 2 (1968), 89-101.

\bibitem{MFT61} M. F. Timan, \textit{Best approximation and modulus of
smoothness of functions prescribed on the entire real axis}, Izv. Vyssh.
Uchebn. Zaved. Mat., 1961, 6, 108-120.

\bibitem{T81} R. Taberski, \textit{Approximation by entire functions of
exponential type}, 1981, Demonstr. Math. 14 (1981), 151-181 .

\bibitem{T86} R. Taberski, \textit{Contributions to fractional calculus and
exponential approximation}, 1986, Funct. Approximatio, Comment. Math. 15
(1986), 81-106 .

\bibitem{TT} P.-Pych. Taberska, R. Taberski, \textit{Exponential
approximation in the norms and semi-norms}, Pliska Studia Mathematica
Bulgarica, Vol. 11, No 1, (1991), 94-101.

\bibitem{HT83} H. Triebel, \textit{Theory of Function Spaces,} Monographs in
Math. Vol. 78, Birkhauser Verlag, Basel, 1983.

\bibitem{TB04} R.M. Trigub, E.S. Belinsky, \textit{Fourier Analysis and
Approximation of Functions}, Kluwer-Springer, 2004, 586 pp.

\bibitem{vak1} S. B. Vakarchuk, \textit{Exact constant in an inequality of
Jackson type for }$\mathit{L}_{2}$\textit{-approximation on the line and
exact values of mean widths of functional classes}, East J. Approxim., 10,
No. 1-2, 27-39 (2004).

\bibitem{vak2} S. B. Vakarchuk, \textit{On some extremal problems of the
approximation theory of functions on the real axis. I}, Ukr. Mat. Visn., 9,
No. 3, 401-429 (2012).

\bibitem{vak3} S. B. Vakarchuk, \textit{On some extremal problems of the
approximation theory of functions on the real axis. II,} Ukr. Mat. Visn., 9,
No. 4, 578-602 (2012).

\bibitem{CNV} C. N. Vasil'ev, \textit{Jackson inequality in }$\mathit{L}_{2}$%
\textit{(}$\mathit{R}^{N}$\textit{) with generalized modulus of continuity},
in: Proc. of the Institute of Mathematics and Mechanics, Ural Division of
the Russian Academy of Sciences [in Russian], 16, No. 4 (2010), pp. 93-99.

\bibitem{yeh} J. Yeh, \textit{Real analysis: theory of measure and
integration}, 2nd ed., 2006

\bibitem{yeydmi11} Y. E. Yildirir and D. M. Israfilov, \textit{Simultaneous
and converse approximation theorems in weighted Lebesgue spaces}, Math.
Inequal. Appl., 14 (2011), No: 2, 359-371.
\end{thebibliography}
\end{document}